\documentclass[11pt, twoside]{article}
\usepackage{amsmath,amsthm,amssymb}
\usepackage{times}
\usepackage{enumerate}
\usepackage{indentfirst, amsfonts, amsmath, amsthm, amssymb, amscd}
\usepackage{amsmath,amsfonts,amscd,bezier}
\usepackage{graphicx}
\usepackage{color}
\usepackage[toc]{appendix}
\usepackage{marginnote}
\usepackage{hyperref}
\usepackage[active]{srcltx}
\usepackage[normalem]{ulem}
\usepackage{authblk}
\usepackage{comment}
\usepackage{color}
\usepackage{xcolor}
\usepackage[percent]{overpic} 

\numberwithin{equation}{section}

\newcommand{\norm}[1]{\left\lVert#1\right\rVert}
\newcommand{\ve}{\varepsilon}
\newcommand{\rarrow}{\rightarrow}

\newcommand{\dd}{\mathrm{d}} 
\newcommand{\N}{\mathbb{N}}
\newcommand{\R}{\mathbb{R}}
\newcommand{\mrm}{\mathrm}

\usepackage{thmtools}
\usepackage{thm-restate}

\newtheorem{theorem}{Theorem}[section]

\newtheorem{proposition}[theorem]{Proposition}
\newtheorem{lemma}[theorem]{Lemma}
\newtheorem{definition}[theorem]{Definition}
\newtheorem{thmx}{Theorem} 

\newtheorem{remark}[theorem]{Remark}

\makeatletter
\def\author#1{\gdef\autrun{\def\and{\unskip, }#1}\gdef\@author{#1}}

\makeatother

\newcommand{\Addresses}{{
  \bigskip
  \footnotesize

 Y. de Jesus, \textsc{Université du Luxembourg, Esh-sur-Alzette, Luxembourg.} \par\nopagebreak
  \textit{E-mail}: \texttt{ygor.dejesus@uni.lu}

M. Espitia, \textsc{Universidad Pontificia Bolivariana, Monteria, Colombia.} \par\nopagebreak
  \textit{E-mail}: \texttt{maresno@gmail.com}

G. Ponce, \textsc{Departamento de Matem\'atica, IMECC-UNICAMP Campinas-SP, Brazil.} \par\nopagebreak
  \textit{E-mail}: \texttt{gaponce@unicamp.br}

}}
                           
\begin{document}

\title{Homoclinic classes for flows: ergodicity and SRB measures}

\author{Ygor de Jesus\footnote{Ygor de Jesus was financially supported by Fundação de Amparo à Pesquisa do Estado de São Paulo (FAPESP - Brazil) through process number 2021/02913-0 and 2023/05100-5.}, Marcielis Espitia, Gabriel Ponce}
  

\date{}


\maketitle


\begin{abstract}
In this work we study homoclinic classes for some classes of flows, extending to flows certain results previously obtained, by other authors, for diffeomorphisms. More precisely, we prove that given a conservative $C^2$ flow $\varphi_t:M\to M$ over a smooth compact Riemannian manifold $M$, if both the stable and unstable homoclinic classes of a certain periodic hyperbolic orbit $\gamma$ have positive Lesbegue measure, then their intersection is an ergodic component of the flow. The same results are valid for invariant SRB measures.
\end{abstract}

\setcounter{tocdepth}{2}
\tableofcontents

\section{Introduction}

One of the most powerful tools developed in the theory of dynamical systems during the last century is now known as the "Hopf Argument." This classical argument is a cornerstone in understanding the \textit{ergodicity} of systems exhibiting hyperbolicity and combines geometric structures, such as the existence of stable and unstable foliations, with global statistical properties of the system. The classical definition of an \textit{ergodic measure} states that an ergodic system cannot be decomposed into two ``meaningful" subsystems that do not interact, that is, any invariant set must have either full measure or zero measure. Among the several equivalent definitions of ergodicity, one particularly relevant to understanding the intuition behind the \textit{Hopf Argument} concerns the behavior of the well-known \textit{Birkhoff Averages} of continuous functions, which must be almost everywhere constant in ergodic systems.

Initially, Hopf's argument applied to systems with a property known as \textit{uniform hyperbolicity}, characterized by the existence of complementary stable and unstable directions in which the system contracts and expands uniformly, respectively. Classical examples of such systems include the famous \textit{Arnold's Cat Map}, defined by the action of the matrix $A=\left(\begin{array}{cc}
    2& 1 \\
    1 &1 
\end{array}\right)$ on $\mathbb{T}^2$ and the \textit{geodesic flow} associated with negatively curved Riemannian metrics. Other forms of hyperbolicity also arise naturally in various systems. In this article, we focus on \textit{non-uniformly hyperbolic flows}, and consequently we rely on several results of Pesin's Theory. 
Understanding ergodicity for these systems often involves decomposing the system into fundamental pieces where detailed analysis can be carried out. Identifying these pieces is a challenging problem. While one might expect a version of the Hopf Argument to apply within these ``fundamental pieces'',  extending Hopf's ideas to systems lacking uniform rates of contraction or expansion is neither straightforward nor trivial.

In \cite{HHTUcriteria}, the authors introduced the notion of \textit{ergodic homoclinic classes} for diffeomorphisms, providing an elegant description of the potential ``fundamental pieces'', namely, the ergodic components of the diffeomorphism. Inspired by their work, we study an analogous decomposition for flows, i.e., the action on a manifold induced by flows generated by certain vector fields and extend their results. The description of homoclinic classes is as follows (see Section \ref{preliminaries} for details): consider the flow $\varphi_t: M\rightarrow M$ generated by a vector field $X$ and $\gamma$ a closed periodic hyperbolic orbit for $\varphi_t$, i.e. there exists $T>0$ such that $\varphi_T(p)=p$ for all $p\in \gamma$ and a decomposition $T_\gamma M=E^s\oplus \mathbb{R}X\oplus E^u$, with $E^s$ contracted by $D\varphi_t$ and $E^u$ expanded by $D\varphi_t$. In this context we can consider the stable and unstable Pesin's manifolds of the orbit $\gamma$, denoted by $W^\tau(\gamma)$ $\tau=s,u$. Then, the stable (resp. unstable) homoclinic class of $\gamma$, which we denote by $\Lambda^s(\gamma)$ (resp. $\Lambda^s(\gamma)$), is the set of points $x$ with well defined stable (resp. unstable) Pesin's manifold and such that $W^s(x)$ (resp. $W^u(x)$) transversely intersects $W^u(\gamma)$ (resp. $W^s(\gamma)$). All along this paper $M$ will denote a smooth closed connected Riemannian manifold, that is, it is a boundaryless smooth compact Riemannian manifold.
We obtain the following result which is an extension of \cite[Theorem A]{HHTUcriteria}:

\begin{thmx}\label{theo:A}
\label{criteria for ergodicity}
Let $\varphi_t:M \to M$ be a $C^{2}$-flow and let $m$ be a smooth $\varphi_t$-invariant measure. If, for a certain hyperbolic non-singular periodic orbit $\gamma$, we have $m(\Lambda^s(\gamma))\cdot m(\Lambda^u(\gamma))>0$, then:
\begin{itemize}
\item[i)] $\Lambda(\gamma) \circeq \Lambda^u(\gamma) \circeq \Lambda^s(\gamma)$,
\item[ii)] $\Lambda(\gamma)$ is a hyperbolic ergodic component for $\varphi_t$.
\end{itemize}
\end{thmx}

\begin{remark}
\label{orbit not singular}
    Observe that if $\gamma$ is a singular periodic orbit then $\text{dim}(W^s(\gamma)) + \text{dim}(W^u(\gamma))-1 = \text{dim}(M) $, consequently, if $x\in \Lambda^s(\gamma)$ the transverse intersection of $W^s(x)$ and $W^u(\gamma)$ yields $\text{dim}(W^s(x))\geq \text{dim}(M) - \text{dim}(W^u(\gamma))$. Analogously for $x\in \Lambda^u(\gamma)$. Consequently, if $x\in \Lambda^s(\gamma)\cap \Lambda^u(\gamma)$ then $\text{dim}(W^s(x))+\text{dim}(W^u(x)) = \text{dim}(M)$ which implies that $\{x\}$ is a singular orbit as well. In this case, even if $\Lambda(\gamma)$ has positive measure, it is not an ergodic component since $\varphi_t |\Lambda(\gamma)  = \text{Id}|\Lambda(\gamma)$.
\end{remark}

An essential property of smooth measures for the proof of the previous theorem is \textit{Absolute Continuity} along Pesin manifolds (check Definition \ref{absolute continuity}). This property is also satisfied by the so called \textit{SRB measures} introduced by Sinai, Ruelle and Bowen in the 70s. \textit{SRB measures} are known as the invariant measures that are most compatible with a volume measure for non conservative systems and plays an important role in the general ergodic theory. With a few extra observations on the proof of Theorem \ref{theo:A} we show that the same result is valid for general invariant SRB measures.

\begin{thmx}
\label{homoclinic class for SRB}
    Let $\varphi_t: M\rightarrow M$ be a $C^2$-flow and $\mu$ a SRB measure. If either $\mu(\Lambda^s(\gamma))\cdot \mu(\Lambda^u(\gamma))>0$ or $m(\Lambda^s(\gamma))\cdot \mu(\Lambda^u(\gamma))>0$ for certain hyperbolic non-singular orbit $\gamma$, then $\Lambda^u(\gamma)\subset \Lambda^s(\gamma)$ $\mod 0$ (for $\mu)$. In addition, $\mu$ restricted to $\Lambda(\gamma)$ is ergodic and hyperbolic physical measure.
\end{thmx}

\noindent The next results deal with the existence of a homoclinic class related to a SRB measure in some cases and also provide us a criterion for obtaining uniquiness of SRB measures from homoclinic classes.

\begin{thmx}
\label{SRB ergodic components has full measure homoclinic class}
    Let $\varphi_t: M\rightarrow M$ be a $C^2$-flow over a compact manifold and $\mu$ a regular, hyperbolic and SRB invariant measure. Then for every ergodic component $\nu$ of $\mu$ there exists a hyperbolic closed orbit $\gamma$ such that $\nu(\Lambda(\gamma))=1$.
\end{thmx}
\begin{thmx}
\label{SRB measures with full measure homoclinic class are equal}
    Let $\varphi_t: M\rightarrow M$ be a $C^2$-flow over a compact manifold and $\mu$ and $\nu$ two ergodic SRB measures. Suppose that there exist a periodic hyperbolic orbit $\gamma$ such that $\mu(\Lambda(\gamma))=\nu(\Lambda(\gamma))=1$. Then, $\mu=\nu$.
\end{thmx}
\noindent We highlight that our results do not state conditions for the existence of such measures. 

In general, it is highly non-trivial to obtain known results for diffeomorphisms in the flow setting by two main reasons: perturbations results for diffeomorphisms are not, in general, applicable for flows and in hyperbolic dynamics there is no contraction nor expansion in the flow direction, hence many of the arguments which depend on this kind of structure are not adaptable. In the same direction of our propose we may cite the recent work by  Hasselblatt and Fisher \cite{fisher2022accessibility}, where the authors need to overcome many difficulties to obtain analogous results from \cite{avila2022symplectomorphisms} in the flow setting. The other way around is also tremendously difficult in many cases, as an example we cite the celebrated work by Avila \cite{Avila2010regularization} on the regularization of conservative diffeomorphisms. It is worth mentioning that flows with some kind of non-uniform hyperbolic behavior appear naturally in dynamical systems, in particular in the context of geodesic flows for non-positively curved Riemannian metrics. Therefore, it is extremely important to develop tools to study such systems and it illustrates well the power of the results presented by this article.

The article is structured as follows: in Section \ref{preliminaries} we present the substantial background material we are going to use, such as the precise definitions and results we are going  to need for our development, in Section \ref{section typical points} we prove some essential preliminary results for the proofs of Theorems \ref{criteria for ergodicity} and \ref{homoclinic class for SRB}. The proof of such results are done in Sections \ref{proof of theorem criteria} and \ref{proof ergodicity SRB}, respectively. Finally, Sections \ref{proof SRB ergodic componets has full measure} and \ref{proof SRB measures with full measure homoclinic class are equal} are devoted to the proofs of Theorems \ref{SRB ergodic components has full measure homoclinic class} and \ref{SRB measures with full measure homoclinic class are equal}, respectively. 

\section{Preliminaries}
\label{preliminaries}
As pointed out in the introductory section, $M$ will denote a closed connected Riemannian manifold of dimension $\dim M=n$, $n\geq 3$. Given a $C^1$ vector field on $M$, the $C^1$ flow generated by $X$ will be denoted by  $\varphi^X=\{\varphi^X_{t}\}_{t\in \mathbb{R}}$ (or just $\varphi$ if there is no confusion). 

Let us fix some usual ergodic theory nomenclature and definitions that we are going to use: a set $A\subset M$ is called \textit{$\varphi-$ invariant} if $\varphi_t(A)=A$, for every $t \in \mathbb{R}$. A Borel probability measure $\mu$ is called $\varphi-$invariant if $\mu(\varphi_t(A))=\mu(A)$, for every $t\in \mathbb{R}$ and every Borel set $A$. A $\varphi-$invariant measure is called \textit{ergodic} if $\mu(A)=1$ or $\mu(A)=0$, for every invariant measurable set $A$. Given a $C^1$ vector field $X$ on $M$, an invariant measure $\mu$ is said to be \textit{regular} if $\mu(\text{Sing}(X))=0$, where $\mathrm{Sing}(X)$ denotes the set of points in $M$ where $X$ vanishes.

\subsection{Nonuniform hyperbolicity}
\label{preliminaries nonuniformm hyperbolicity}

Let us start by defining a key concept to the study of Nonuniform hyperbolic systems called \textit{Lyapunov Exponents}. Indeed, \textit{Lyapunov Exponents} are fundamental tools in the study of dynamical systems as they quantify the exponential rates at which nearby trajectories diverge or converge over time. Let $\varphi_t\colon M\rarrow M$ be a $C^1$-flow
\begin{definition}
\label{definition lyapunov exponents}
    For each $x\in M$ and $v\in T_xM$ we define the \textit{Lyapunov Exponent} associated to $(x,v)$ as
    \[
    \chi(x,v)=\limsup_{|t|\rightarrow \infty }\frac{1}{|t|}\log\norm{(D\varphi_t)_xv},
    \]
    when the above limit exists.
\end{definition}
\noindent For each $x\in M$, it can be proved that $\chi(x,\cdot)$ attains finitely many distinct values on the tangent space at $x$, so we denote them by
\[
\chi_1(x)< \cdots <\chi_{l(x)}(x).
\]
At any point $x\in M$, the Lyapunov Exponents define a filtration of the tangent space at $x$, i.e. a sequence of subspaces $\{W_i(x)\}_{i=0}^{l(x)}$, defined by
\[
W_i(x)=\{v\in T_xM: \chi(x,v)\leq \chi_i(x)\},
\]
which satisfies $\{0\}=W_0(x)\subsetneq W_1(x)\subsetneq \cdots \subsetneq W_{l(x)}(x)=T_xM$. The multiplicity of each Lyapunov Exponent is the number $n_i(x)=\dim W_{i+1}(x)-\dim W_i(x)$.


The first natural question that arises about Lyapunov Exponents is why or when the limit in Definition \ref{definition lyapunov exponents} exists. This question is answered by one of the most fundamental theorems in the theory of nonuniform hyperbolic systems, and it is a consequence of the Multiplicative Ergodic Theorem by Valery Oseledets \cite{OSELDETS68}:
\begin{theorem}[Oseledets' Decomposition]
   Let $\varphi_t\colon M\rarrow M$ be a $C^1$-flow on a closed Riemannian manifold $M$. There exists an invariant set $\mathcal R \subset M$ of full measure with respect to any invariant Borel probability measure $\mu$, such that for every $x \in \mathcal R$:

\begin{enumerate}
    \item  The tangent space $T_xM$ admits a splitting:
    \[
    T_xM = \bigoplus_{i=1}^{l(x)} E_i(x),
    \]
    where $E_i(x)$ are called the Oseledets subspaces and $l(x)$ is the number of distinct Lyapunov exponents at $x$.

    \item  There exist real numbers $\chi_1(x) < \cdots < \chi_{l(x)}(x)$ such that for any $v \in E_i(x) \setminus \{0\}$:
    \[
    \lim_{|t| \to \infty} \frac{1}{t} \log \|D\varphi_t(x)v\| = \chi_i(x).
    \]

    \item The subspaces $E_i(x)$ are invariant under the derivative of the flow:
    \[
    D\varphi_t(x)(E_i(x)) = E_i(\varphi_t(x)).
    \]

    \item  For any disjoint subsets $I, J \subset \{1, \ldots, l(x)\}$, let $E_I(x) = \bigoplus_{i \in I} E_i(x)$ and $E_J(x) = \bigoplus_{j \in J} E_j(x)$. Then:
    \[
    \lim_{|t| \to \infty} \frac{1}{t} \log \angle(D\varphi_t(x)E_I(x), D\varphi_t(x)E_J(x)) = 0.
    \]

    \item The growth rate of the determinant of the derivative of the flow is given by:
    \[
    \lim_{|t| \to \infty} \frac{1}{t} \log \det(D\varphi_t(x)) = \sum_{i=1}^{l(x)} \chi_i(x) \dim E_i(x).
    \]

    \item The functions $\chi_i(x)$ and $\dim E_i(x)$ are $\varphi_t$-invariant. In particular, if $\mu$ is ergodic, then $\chi_i(x)$ and $\dim E_i(x)$ are $\mu$-almost everywhere constant.
\end{enumerate}
\end{theorem}
\begin{definition}
The set $\mathcal{R}$ is called the set of \textit{Regular Points}.
\end{definition}
\noindent We now outline essential results from Pesin's Theory, focusing on their application to flows, which is our main interest in this dissertation. All the presented results have analogous for diffeomorphisms.




\begin{definition}
   Let $\mu$ be an ergodic measure. If, there exists a $\mu$-full measure set $\tilde{\mathcal{R}}\subset\mathcal{R}$ such that for $x\in \tilde{\mathcal{R}}$ we have: the subspace $E_0(x)$, generated by the vectors with zero Lyapunov exponents, satisfies $E_0(x) = \R X$, where $X$ is the vector field that generates the flow $\varphi_t$. Then the flow is said to be \textit{nonuniformly hyperbolic} on $\tilde{\mathcal{R}}$, and $\mu$ is called a \textit{hyperbolic measure} on $\tilde{\mathcal{R}}$. 
\end{definition}

For the points $x\in \tilde{\mathcal{R}}$ notice that there exists $k\in \N$ such that $\chi_k(x)<0<\chi_{k+1}(x)$ and we can define

\[
    E^s(x)=\bigoplus_{\chi_i(x)<0}E_i(x) \,\,\,\,\,\,\mbox{    and   }\,\,\,\,\,\,E^u(x)=\bigoplus_{\chi_i(x)>0}E_i(x).
\]
We have the following results that summarize the properties of these spaces

\begin{theorem}\emph{\cite[Theorem 2.1.3]{BP2002}}
    The following properties hold for $x\in \tilde{\mathcal{R}}$:
    \begin{enumerate}
        \item $E^s(x)$ and $E^u(x)$ depend measurably on $x\in \tilde{\mathcal{R}}$.
        \item We have the splitting $T_xM=E^s(x)\oplus \R X\oplus E^u(x)$.
        \item $(D\varphi_t)_xE^{s,u}(x)=E^{s,u}(\varphi_t(x))$, for all $t\in \R$.
    \end{enumerate}
    There exists $\ve_0>0$ and Borel functions $C(x,\ve)>0$ and $K(x,\ve)>0$ such that for all $x\in \tilde{\mathcal{R}}$ and $0<\ve\leq \ve_0$
    \begin{enumerate}
        \item For all $v\in E^s(x)$ and $t>0$ it holds that
        \[
        \norm{(D\varphi_t)_xv} \leq C(x,\ve)\mrm{e}^{(\chi_k+\ve)t}\norm{v}.
        \]
        \item For all $v\in E^u(x)$ and $t<0$ it holds that
        \[
        \norm{(D\varphi_t)_xv} \leq C(x,\ve)\mrm{e}^{(\chi_k-\ve)t}\norm{v}.
        \]
        \item The angle between the spaces $E^s(x)$ and $E^u(x)$ is bounded from below by the function $K(x,\ve)$.
        \item The functions $C(x,\ve)$ and $K(x,\ve)$ are not $\varphi_t$-invariant in general but they satisfy
        \[
        C(\varphi_t(x),\ve)\leq C(x,\ve)\mrm{e}^{\ve|t|}
        \]
        and
         \[
        K(\varphi_t(x),\ve)\geq K(x,\ve)\mrm{e}^{-\ve|t|}
        \]
    \end{enumerate}
\end{theorem}

From now on, let us denote $\tilde{\mathcal{R}}$ by $\mathcal{R}$, just to simplify our notation. We can provide a more detailed description of the structure of a nonuniform hyperbolic set $\mathcal{R}$ of Regular Points. Let $\varepsilon > 0$ and $l > 0$, we define the  \textit{Pesin block} (of level $l$) as the following set:  
\[
\mathcal{R}^{i,j}_{\varepsilon,l} = \left\{x \in \mathcal{R} \, : \, C(x, \varepsilon) \leq l, \, K(x, \varepsilon) \geq \frac{1}{l}, \dim{E^s(x)}=i\text{ and } \dim{E^u(x)}=j\right\}.
\]
Sometime we will simply consider
\[
\mathcal{R}^l:=\left\{x \in \mathcal{R} \, : \, C(x, \varepsilon) \leq l, \, K(x, \varepsilon) \geq \frac{1}{l}\right\}.
\]
This set has the following fundamental properties:  
\begin{enumerate}
\item $
    T_x M=\bigoplus_{\lambda<0} E_\lambda(x)\oplus E^c(x)\bigoplus_{\lambda>0} E_\lambda$, where $E^c(x)=E_0(x)\oplus X(x)$.
    \item $\mathcal{R}^{i,j}_{\varepsilon,l}\subset \mathcal{R}^{i,j}_{\varepsilon,l+1}$;
    \item for any $t \in \mathbb{R}$, $\varphi_t(\mathcal{R}^{i,j}_{\varepsilon,l}) \subset \mathcal{R}^{i,j}_{\varepsilon,l'}$, where $l' = l \exp(|t| \varepsilon)$;
    \item the subspaces $E^s(x)$ and $E^u(x)$ depend continuously on $x \in \mathcal{R}^{i,j}_{\varepsilon,l}$.
    \item There exists $\delta:=\delta(i,j,\ve,l)>0$ such that for any $x\in \mathcal{R}^{i,j}_{\varepsilon,l}$, $W^s(x)$ and $W^u(x)$ contain open disks containing $x$ of dimension $i$ and $j$, respectively, and uniform diameter $\delta$. They are called, respectively, local stable and unstable manifolds of $x$ and are denoted, respectively, by $W^s_{\text{loc}}(x)$ and $W^u_{\text{loc}}(x)$.
\end{enumerate}

Now we present the Stable Manifold Theorem of Pesin, a fundamental result in the theory of nonuniformly hyperbolic dynamical systems, particularly in the context of flows.
\begin{theorem}\cite[Stable manifold for flows]{BP2002}
    Let $\mathcal{R}$ be a nonuniformly hyperbolic set for a smooth flow $\varphi_t$. Then, for every $x \in \mathcal{R}$, there exist a local stable manifold $W_{loc}^s(x)$ and a local unstable manifold $W_{loc}^u(x)$ such that:
    \begin{enumerate}
        \item $x \in W_{loc}^s(x)$, $T_xW_{loc}^s(x) = E^s(x)$, and for any $y \in W_{loc}^s(x)$ and $t >0$, 
        \[
        d(\varphi_t(x), \varphi_t(y)) \leq T(x)\lambda^t e^{\varepsilon t}d(x, y),
        \]
        where $T : \mathcal{R} \to (0, \infty)$ is a Borel function satisfying, for all $s \in \mathbb{R}$,  
        \[
        T(\varphi_s(x)) \leq T(x)e^{10\varepsilon|s|}.
        \]
        \item $x \in W_{loc}^u(x)$, $T_xW_{loc}^u(x) = E^u(x)$, and for any $y \in W_{loc}^u(x)$ and $t < 0$, 
        \[
        d(\varphi_t(x), \varphi_t(y)) \leq T(x)\lambda^{-t} e^{\varepsilon |t|}d(x, y).
        \]
        \end{enumerate}
\end{theorem}

\begin{definition}
    The manifolds $W_{loc}^{s,u}(x)$ are called the (un)stable Pesin's manifolds.
\end{definition}

 For a regular point $x\in M$, denote its orbit by $\gamma:=\{\varphi_{t}(x)\}_{t\in \R}$. So we can define its \textit{global (un)stable manifold} as
\[
W^{s}(x)=\bigcup_{t \geq 0}\varphi_{-t}(W_{loc}^{s}(\varphi_t(x)), \quad W^{u}(x)=\bigcup_{t \geq 0}\varphi_{t}(W_{loc}^{s}(\varphi_{-t}(x)).
\]
We also define for every $x\in \mathcal{R}$ its \textit{global weakly stable manifolds} and \textit{global weakly unstable manifolds} at $x$ by
\[
W^{ws}(x)=\bigcup_{t \in \mathbb{R}}W^{s}(\varphi_t(x)), \quad W^{wu}(x)=\bigcup_{t \in \mathbb{R}}W^{u}(\varphi_t(x))
\]
For a hyperbolic orbit $\gamma=\{\varphi_t(x)\}_{t\in \R}$, we denote also the global weakly stable and unstable manifolds in $x$ by $W^{s}(\gamma)$ and $W^{u}(\gamma)$, respectively.

We are now finally able to define the main object of this paper, the \textit{Ergodic Homoclinic classes for flows}: given a periodic hyperbolic orbit $\gamma$, we define the \textit{stable and unstable  homoclinic class of $\gamma$} as follows:
\begin{equation}
\label{stable homoclinic class}
\Lambda^s(\gamma) = \{x \in M : x \text{ is a regular point and } W^s(x) \pitchfork W^u(\gamma) \neq \emptyset\},
\end{equation}
and
\begin{equation}
\label{unstable homoclinic class}
\Lambda^u(\gamma) = \{x \in M : x \text{ is a regular point and } W^u(x) \pitchfork W^s(\gamma) \neq \emptyset\}.
\end{equation}
\begin{definition}
Here, the symbol "$\pitchfork$" means the following: given two submanifolds of $M$, say $N_1$ and $N_2$, then $N_1\pitchfork N_2\neq \emptyset$ if
\begin{enumerate}
    \item $N_1\cap N_2\neq \emptyset$.
    \item For every $z\in N_1\cap N_2$, one has $T_zN_1+T_zN_2=T_zM$.
\end{enumerate}
In this case, we say that $N_1$ and $N_2$ have a transverse intersection.
\end{definition}

For the above definitions, we are not assuming that $Sing(X)=\emptyset$ nor that $\gamma$ is not a singular orbit; it could be that $X|_\gamma =0$. However, as stated in Remark \ref{orbit not singular}, our result does not work in this case. We also highlight that $\Lambda^s(\gamma)$ is $s$-saturated and $\Lambda^u(\gamma)$ is $u$-saturated, i.e. for every $x\in \Lambda^{s,u}(\gamma)$ we have $W^{s,u}(x)\subset \Lambda^{s,u}(\gamma)$. We also easy to see that both sets are $\varphi_t$-invariant.
The \textit{Ergodic homoclinic class of $\gamma$} is then defined as the $s,u$-saturated and $\varphi_t$-invariant set:
\begin{equation}
    \Lambda(\gamma) = \Lambda^s(\gamma) \cap \Lambda^u(\gamma).
\end{equation}

\begin{figure}[h!]
\begin{overpic}[width=0.9\textwidth,tics=10,grid=false]{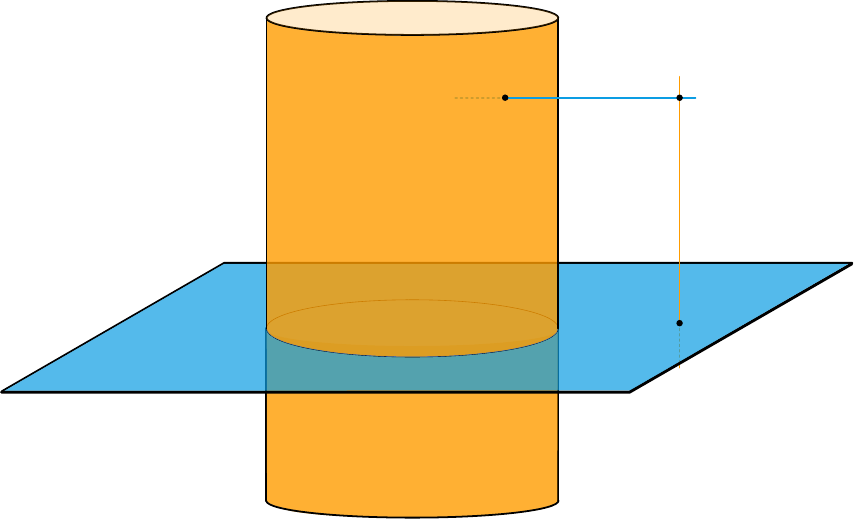}
\put(60,18){$\gamma$}
\put(90,32){$W^s(\gamma)$}
\put(22,57){$W^u(\gamma)$}
\put(77.5,47.5){$z$}
\put(82,48){$W^s(z)$}
\put(78,53){$W^u(z)$}
\end{overpic}
\centering
\caption{Illustration of the local dynamics around a point $z\in \Lambda(\gamma)$.}
\end{figure}

A prior, there is no reason to believe that the above set is not empty. Indeed, it is precisely the result obtained by Theorem \ref{criteria for ergodicity} to state conditions that guarantee that $\Lambda(\gamma)$ is nonempty.

We now discuss the main property we are going to use to prove our results, the \textit{Absolute Continuity} of the stable and unstable partitions by $W^s$ and $W^u$. Essentially, it says that we can transfer information from a lamination $W^u(x)$ to another $W^u(y)$ by flowing through $W^s(x)$ without losing too much information from the measure theoretical point of view. Below, we are going to make precise the expression "flowing through $W^s(x)$" as the holonomy maps. \textit{Absolute Continuity} is the core property of the Hopf's Argument.

\begin{definition}
    A partition $\xi$ of $M$ is called a \textit{Measurable Partition} if the quotient space $M/\xi$ can be generated by a countable collection of measurable sets. 
\end{definition}

\noindent When $M$ is a Lebesgue space, the quotient space $M/\xi$ obtained from a measurable partition $\xi$ is also a Lebesgue space \cite{Rohlin52} (see also the chapter 15 of \cite{COUDENE16ergodic}). Measure partitions are important from the measure-theoretical point of view since they allow us to decompose a measure into measures along the "pieces" of the partition:
\begin{proposition}
For any measurable partition $\xi$ of a Lebesgue space $(M, \mathcal{B}, m)$, there exists a canonical system of \textit{Conditional Measures} $m^\xi_x$ with the following properties:
\begin{enumerate}
    \item The conditional measures $m^\xi_x$ are defined in $\xi(x)$, the partition element containing $x$.
    \item For any $A \in \mathcal{B}$, the set $A \cap \xi(x)$ is measurable in $\xi(x)$ for almost all $\xi(x) \in M/\xi$.
    \item The mapping $x \mapsto m^\xi_x(A \cap \xi(x))$ is measurable, and the measure $m$ satisfies:
    \[
    m(A) = \int_{M/\xi} m^\xi_x(A \cap \xi(x)) \, dm_T,
    \]
    where $m_T$ denotes the quotient measure in $M/\xi$.
\end{enumerate}
This canonical system of conditional measures is unique (mod 0) for any measurable partition. Conversely, if a canonical system of conditional measures exists for a partition, the partition must be measurable.
\end{proposition}
\begin{proof}
    See the Chapter $5$ of \cite{KerleyViana}.
\end{proof}

\begin{definition}A measurable partition $\xi$ is said to be \textit{subordinate} to the unstable partition $W^u$ if, for $m$-almost every $x \in M$, the following conditions hold:
\begin{enumerate}
    \item  $\xi(x) \subset W^u(x)$
    \item  $\xi(x)$ contains a neighborhood of $x$ that is open in the topology of $W^u(x)$.
\end{enumerate}
\end{definition}

\noindent Given a Riemannian metric $g$ for $M$, it induces Riemannian measures on each $W^\tau(x)$ which we are going to denote by $\lambda_x^{\tau}$, $\tau=s,u$.
\begin{definition}
\label{absolute continuity}
    We say that a measure $\nu$ has absolutely continuous conditional measures with respect to the unstable (resp. stable) manifolds if for any measurable partition $\mathcal{P}$ subordinated to $W^u$ (resp. $W^s$) we have $\nu_x^{\mathcal{P}}<<\lambda_x^u$ (resp. $\lambda_x^s$) for $\nu$-a.e. $x$.
\end{definition}

We are now able to make sense to the previously used expression "transfer information from a lamination $W^u(x)$ to another $W^u(y)$" which is the notion of \textit{Holonomy Map}: let $\mathcal{R}^{i,j}_{\varepsilon,l}$ be a Pesin block. For each point $x \in \mathcal{R}^{i,j}_{\varepsilon,l}$ that admits a negative (resp. positive) Lyapunov exponent, there exists the local Pesin stable (resp. unstable) manifold, denoted by $W^s_{\text{loc}}(x)$ (resp. $W^u_{\text{loc}}(x)$), with diameter at least $\delta>0$. For such $x\in\mathcal{R}^{i,j}_{\varepsilon,l}$ and given two transversal disks $D_1$ and $D_2$ to $W^u_{\text{loc}}(x)$ (resp. $W^s_{\text{loc}}(x)$) and close to each other, the stable (resp. unstable) holonomy map $h^{s,u}\colon D_1'\subset D_1\rarrow D_2$ is defined as:

\[
h^{s,u}(z):=W^{s,u}(z)\cap D_2,
\]
where
\[
D_1':=\{z\in D_1\cap R_{\ve,l}^{i,j}: W^{s,u}_{loc}(z)\pitchfork D_{1,2}\neq \emptyset \}
\]
Essentially, for a point on $D_1$ we consider (when it exists) its stable (resp. unstable) Pesin manifold and its intersection with $D_2$. Analogously, we can define the weak-stable and weak-unstable holonomy maps by switching $W^{s,u}_{loc}$ by $W^{ws,wu}_{loc}$.

\begin{theorem}\emph{\cite[Theorem 4.3.1]{BP2002}}
    The (weak) holonomy maps ($h^{ws,wu}$) $h^{s,u}$  are measurable and absolutely continuous with respect to the Lesbegue measures induced on $D_1$ and $D_2$, that is they send sets of zero Lesbegue measure into sets of zero Lesbegue measure.
\end{theorem}

\begin{figure}[h!]
\begin{overpic}[width=0.8\textwidth,tics=10,grid=false]{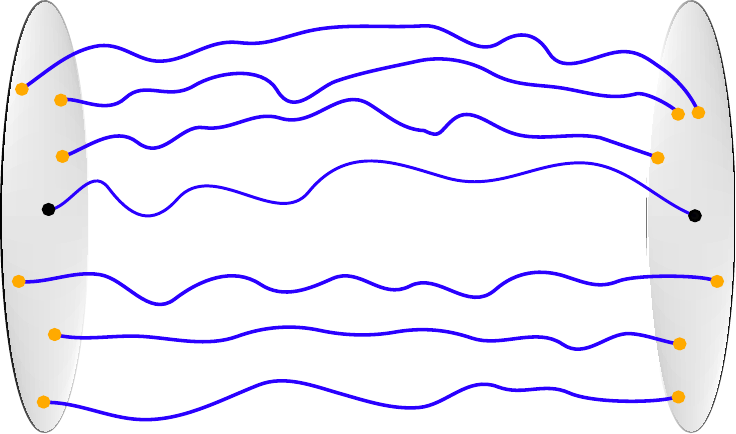}
\put(6,28){$x$}
\put(50,32){$W^s_{\text{loc}}(x)$}
\put(6,37){$z$}
\put(91,37){$h^s(z)$}
\put(5,-3){$D_1$}
\put(93,-3){$D_2$}
\end{overpic}
\centering

\caption{Representation of the stable holonomy map}
\end{figure}

\subsection{Sinai-Ruelle-Bowen measures}
In this subsection, we will recall a particular class of measures called \textit{Sinai-Ruelle-Bowen measures} or only \textit{SRB measures}. These are the measures treated in Theorems \ref{homoclinic class for SRB}, \ref{SRB ergodic components has full measure homoclinic class}, and \ref{SRB measures with full measure homoclinic class are equal}. We do not intend to develop the whole theory about these measures since we are going to use mostly their definition. For a broader treatment, we refer to \cite{YOUNG02}. SRB measures are frequently introduced as "the invariant measures most compatible with volume when volume is not preserved" (see \cite{YOUNG02}). The most important property for us is going to be absolute continuity of conditional measures:

\begin{definition}
   A measure $\nu$ is called an SRB (Sinai-Ruelle-Bowen) measure if it has a positive Lyapunov Exponent at $\nu$-almost every point $x$ and absolutely continuous conditional measures with respect to the unstable manifolds.
\end{definition}
Although other equivalent definitions are possible, the above definition presents precisely the important property that allow us to perform a version of the Hopf Argument in our context. Remember that our results deal with the relations between SRB measures and homoclinic classes for flows, and no further directions are considered in this dissertation. The techniques we use to obtain Theorem \ref{criteria for ergodicity} are extremely similar to those for SRB measure because of the absolute continuity property. Besides that, in some cases it is easier to work with the equivalent definition after the work of Ledrappier-Young \cite{ledrappier1985metric}: absolute continuity with respect to the unstable manifolds is equivalent to Pesin's formula, i.e., it holds that

\[
h_\nu(\varphi_t)=\int \sum_{\lambda(x)>0}\lambda(x)d\nu.
\]

Recall that Theorem \ref{SRB ergodic components has full measure homoclinic class} refers to ergodic components of a metric. This notion is due the classical result \textit{Ergodic decomposition theorem}, which we state below for completeness in the manifold setting. For the proof, we recommend the book of Viana and Oliveira \cite{KerleyViana} Chapter $5$. Remember that given a partition $\mathcal{P}$ of a probability space $(M,\mu)$ into measurable sets, then there exists a canonical structure of probability space $(\mathcal{P},\widehat{\mu})$.

\begin{theorem}[Ergodic decomposition]
    If $M$ is a compact manifold, $f\colon M\rarrow M$ a measurable transformation and $\mu$ a probability measure. Then, there exist a measurable set $M_0$ with full measure, a partition $\mathcal{P}$ of $M_0$ into measurable subsets and a collection of probability measures in $M$, say $\{\mu_P: P\in \mathcal{P}\}$, such that the following hold:
    \begin{enumerate}
        \item $\mu_P(P)=1$, for $\widehat{\mu}$ almost every $P\in \mathcal{P}$.
        \item For any measurable subset $E\subset M$, the map $P\mapsto \mu_P(E)$ is measurable.
        \item For $\widehat{\mu} $almost every $P\in \mathcal{P}$, $\mu_P$ is $f$-invariant and ergodic.
        \item For any measurable subset $E\subset M$, it holds:
        \[
        \mu(E)=\int_{\mathcal{P}}\mu_P(E)\dd\widehat{\mu}(P).
        \]
    \end{enumerate}
\end{theorem}

For the proof of Theorem \ref{SRB ergodic components has full measure homoclinic class}, we are going to use the following flow version of a result by Katok in \cite{katok1980lyapunov} whose proof can be found in \cite{LI&LIU24}:
\begin{theorem}[\cite{katok1980lyapunov,LI&LIU24}]
\label{Katok existence of periodic orbit}
    Let $\mu$ be a regular hyperbolic ergodic measure of a $C^2$ flow $\varphi_t$. Then there exists a hyperbolic periodic orbit $\gamma$ such that $supp(\mu)\subset \overline{\Lambda(\gamma)}$ and $\mu$ is homoclinically related with $\gamma$. In particular, for $\mu$-almost every point $x$, $\varphi_t(W^u(x))$ accumulates on $W^u(\gamma)$ as $t$ goes to infinity.
\end{theorem}


\section{Typical points for smooth and SRB measures}
\label{section typical points}
\noindent As mentioned in the introduction, the strategy of the proofs of Theorems \ref{criteria for ergodicity} and \ref{homoclinic class for SRB} follows the line of Hopf's Argument, which was developed to prove ergodicity of uniform hyperbolic systems. The main difference here is that some important sets (which we will make precise) are not uniformly distributed along Pensin's manifolds. Instead we will see that they are well distributed from the measure point of view but, in general, they are not fully $su-$saturated. 

To make precise our statements, let us introduce some important notions such as Birkhoff's average: for each $f \in L_m^1(M;\mathbb{R})$ the Birkhoff's Ergodic Theorem assures that $m-$almost everywhere the following limits are well-defined
\[
f^\pm(x)=\lim_{T\rightarrow \pm\infty }\frac{1}{T}\int_0^T f(\varphi_t(x))dt.
\]
We also have that $f^+(x)=f^-(x)$ for $m-$almost every $x \in M$ and $f^\pm$ are $\varphi_t-$invariant. Remember that ergodicity is equivalent to $f^+$ being constant almost everywhere. The next Lemmas give us a description of the behavior of Birkhoff's averages along Pesin's manifolds. Essentially we will see that almost every points in the same (un)stable Pesin's manifold present the same Birkhoff average, so we say that the subset of Pesin's manifolds with constant Birkhoff average is well distributed along such manifolds from the measure theoretical point of view.

\begin{lemma}[Typical points for continuous functions]
\label{Typical points for continuous functions}
There exists a $\varphi_t-$invariant set $M_0$, with $m(M_0)=1$, such that for any $f \in C^0(M)$ we have: if $x \in M_0$, then $f^+(x)=f^+(w)$, for any $w \in W^s(x)$ and $m_x^u-$almost every $w \in W^u(x)$.
\end{lemma}
\begin{proof}
The conclusion for the stable manifold is a consequence of continuity. Let $w\in W^s(x)$, in particular $d(\varphi_t(x),\varphi_t(y))\rarrow 0$ as $t\rarrow\infty$. Let us fix any $f\in C^0(M)$, then compactness of $M$ implies that $f$ is uniformly continuous. Therefore, let $\ve>0$ and chose $\delta>0$ such that for all $x,y\in M$, with $d(x,y)<\delta$, it holds that $|f(x)-f(y)|<\frac{\ve}{2}$. Now, let $T_0>0$ be such that $d(\varphi_t(x),\varphi_t(w))<\delta$, for $t\geq T_0$. We get,
\begin{align*}
    \left|\frac{1}{T}\int_0^T(f(\varphi_t(x))-f(\varphi_t(w))\dd t\right|&\leq \frac{1}{T}\int_0^{T_0}|f(\varphi_t(x))-f(\varphi_t(w))|\dd t\\
    &+\frac{1}{T}\int_{T_0}^{T}|f(\varphi_t(x))-f(\varphi_t(w))|\dd t\\
    &\leq \frac{I}{T}+\frac{(T-T_0)}{T}\frac{\ve}{2},
\end{align*}
where
\[
I=\int_0^{T_0}|f(\varphi_t(x))-f(\varphi_t(w))|\dd t <\infty
\]
It implies that for $T>0$ sufficiently large, we get

\[
\left|\frac{1}{T}\int_0^T(f(\varphi_t(x))-f(\varphi_t(w))\dd t\right|<\ve.
\] 
Therefore, we conclude that $f^+(x)=f^+(w)$, for every $f\in C^0(M)$.

Now, the conclusion for the unstable manifold is a little bit more delicate. We proceed as follows: define the following full-measured set (see Theorem 3.2.6 of \cite{KerleyViana}): 
\[
S_0=\{x \in M: f^+(x) = f^-(x) \mbox{ are well-defined}, \mbox{ for all } f\in C^0(M)\}.
\]
\noindent We will prove that $m-$almost every point $x$ of $S_0$ is such that $m_x^u-$almost every $\xi \in W_{loc}^u(x)$ also lies inside $S_0$. In fact, suppose this is not the case. Thus there exists a subset $A\subset S_0$, with $m(A)>0$, such that for any $x \in A$ there exists a set $B_x \subset W_{loc}^u(x)\setminus S_0$, with $m_x^u(B_x)>0$. Without losing generality, we can switch $A$ by $A\cap \mathcal{R}^l$, where $\mathcal{R}^l$ is a Pesin block such that this intersection has positive measure. Let $y \in A$ be a Lesbegue density point (in fact we only need it to be in $\mrm{supp}(m)$), that is
\[
\lim_{\varepsilon \rightarrow 0}\frac{m(A\cap B_\varepsilon(y))}{m(B_\varepsilon(y))}=1
\]
So, for $\varepsilon>0$ small enough we can consider a small disk $T$ inside $A\cap B_\varepsilon(y)$ that is transverse to its points unstable manifolds and the set
\[
B=\bigcup_{x \in T}\tilde{B}_x,
\]
where
\[
\tilde{B}_x=\{w \in B_x: w \in B_\varepsilon(y)\}
\]
satisfies
\[
m(B)=\int_T m_x^u(\tilde{B}_x)dm_T(x)>0.
\]
This is a contradiction because $B$ is a positive measure set outside $S_0$.

\noindent Define the full measure set
\[
S_1=\{x \in S_0: m_x^u-\mbox{almost every } \xi \in W_{loc}^u(x) \mbox{ lies inside } S_0\}.
\]

\noindent If $\xi \in W_{loc}^u(x)$, then for any $x \in M$, we have that $f^+(\xi)=f^+(x)$. In addition if $x \in S_0$ and $\xi \in  W_{loc}^u(x)\cap S_0$, then by continuity $f^+(\xi)=f^-(\xi)=f^-(x)=f^+(x)$. We conclude that $S_1$ consist of points in $S_0$ such that $m_x^u-$almost every $\xi \in W_{loc}^u(x)$ satisfies $f^+(\xi)=f^+(x)$. Since $f$ is continuous, then $f^+$ is constant on $W^s(x)$. Thus, by the invariance of $f^+$ we can find a full measure set $M_0$ with the desired property.
\end{proof}
\begin{remark}
    Notice that $M_0$ is $s-$saturated.
\end{remark}

\begin{lemma}[Typical set for $L^1$ functions]
\label{Typical points for integrable functions}
For any $f \in L^1(M;\mathbb{R})$ there exists a $\varphi_t-$invariant set $\mathcal{T}_f$, with $m(\mathcal{T}_f)=1$, that satisfies: for all $x \in \mathcal{T}_f$ we have $f^+(z)=f^+(x)$ for $m_x^s-$almost every $z \in W^s(x)$ and $m_x^u-$almost every $z \in W^u(x)$. 
\end{lemma}

\begin{proof}
Let $f \in L^1(M;\mathbb{R})$ and consider a sequence $(f_n)_n \subset C^0(M)$ that converges to $f$ in the $L^1-norm$. Then, we have that $(f_n^+)_n$ converges, in the $L^1-norm$, to $f^+$. It implies that there is a subsequence, say $(f_{n_k}^+)_k$, that converges almost everywhere to $f^+$. Set
\[
\mathcal{T}_0=\{x \in M: f_{n_k}^+(x) \rightarrow f^+(x)\}
\]
and define $\mathcal{T}_f:=\mathcal{T}_0\cap M_0$, where $M_0$ is the set of typical points given by the previous lemma. Of course $m(\mathcal{T}_f)=1$. If $x \in \mathcal{T}_f$, then $f_{n_k}^+(x) \rightarrow f^+(x)$ and we also have $f_{n_k}^+(w)=f_{n_k}^+(x)$, for all $w \in W^s(x)$ and $m_x^u-$almost every $w \in W^u(x)$. Since $m_x^s-$almost every $w \in W^s(x)$ lies inside $\mathcal{T}_0$, we see that $f_{n_k}^+(w)=f_{n_k}^+(x)$ converges both to $f^+(w)$ and $f^+(x)$. The result follows.
\end{proof}
\noindent If $\mu$ is an SRB measure, then its conditional measures are absolutely continuous along stable and unstable manifolds. Hence, the following two Lemmas can be proved analogously by using a point  in the support of $\mu$ instead of using a Lesbegue density point.
\begin{lemma}
\label{Typical points for continuous functions and SRB measures}
    Let $\mu$ be a SRB-measure, then there exists a $\varphi_t-$invariant set $M_0$, with $\mu(M_0)=1$, such that for any $f\in C^0(M)$ we have: if $x\in M_0$, then $f^+(x)=f^+(w)$, for any $w\in W^s(x)$ and $m_x^u-$almost every $w\in W^u(x)$.
\end{lemma}
\begin{lemma}
\label{Typical points for L1 functions and SRB measures}
   For any $f \in L^1(M;\mathbb{R})$ there exists a $\varphi_t-$invariant set $\mathcal{T}_f$, with $\mu(\mathcal{T}_f)=1$, that satisfies: for all $x \in \mathcal{T}_f$ we have $f^+(z)=f^+(x)$ for $m_x^s-$almost every $z \in W^s(x)$ and $m_x^u-$almost every $z \in W^u(x)$.  
\end{lemma}

\section{Proof of Theorem \ref{criteria for ergodicity}}
\label{proof of theorem criteria}
We are going to split the proof of Theorem \ref{criteria for ergodicity} into two parts: first we prove that $\varphi|_{\Lambda(\gamma)}$ is an ergodic flow, whichs proof is simpler and illustrates well the idea to prove the first statement $\Lambda^u(\gamma) \circeq \Lambda^s(\gamma)$. 

The general idea is the following: for the second statement we need to prove that Birkhoff averages of continuous functions are almost everywhere constant on the ergodic homoclinic classes. To this end, given two points $x$ and $y$ homophonically related to $\gamma$ we are going to use this relation to transfer the set of typical points from $W^u(y)$ to $W^u(x)$ via stable holonomy. The idea behind the first statement is similar but, we need to deal with Birkhoff averages from just measurable functions.

\subsection{Proof of the second statement}
\noindent For the sake of simplicity, let us fix some reference point $p\in\gamma$. Remember that we are considering the case where $\gamma$ is not a singular orbit, otherwise the results does not hold as stated in Remark \ref{orbit not singular}. 

Define  $\Delta=\Lambda(\gamma)\cap M_0$, where $M_0$ is given by Lemma \ref{Typical points for continuous functions} and $\mathcal{R}$ is the set of regular points. Let $f\in C^0(M)$ and let us prove that $f^+$ is constant on $\Delta$. For each Pesin block $\mathcal{R}_{\ve,l}^{i,j}$ intersecting $\Delta$ in a positive measure set, denote by $\Delta_{\ve,l}^{i,j}$. Inside the Pesin Block the points have uniform size of local stable and unstable Pesin's Manifold, say $\delta>0$. By the Poincaré Recurrence Theorem we can find take $x, y \in \Delta_{\ve,l}^{i,j}$ which return infinitely many times to $\Delta_{\ve,l}^{i,j}$, i.e. there exist sequences $(t_k)_k$ and $(s_l)_l$ such that

\begin{enumerate}
    \item $t_k \rightarrow -\infty$, $x_k=\varphi_{t_k}(x) \in \Delta_{\ve,l}^{i,j}$.
    \item $s_l \rightarrow \infty$, $y_l = \varphi_{s_l}(y) \in \Delta_{\ve,l}^{i,j}$.
\end{enumerate}

For all $q \in M_0$ we can define the $m_q^u-$full measure sets

\[
A_{q}=\{\xi \in W^u(q):f^+(\xi)=f^+(q)\} \subset M_0
\]

\noindent Since $W^s(y_l)\pitchfork W^u(\gamma)$, we can find $l>>1$ big enough such that $d(y_l,W^u(\gamma))<\frac{\delta}{4}$. Because $y_l\in \Lambda(\gamma)$, then $y_l$ is a hyperbolic point with dimensions of stable and unstable equal to those of $p\in \gamma$ and $W^{ws}(y_l)\pitchfork W^u(\varphi_T(p))$ (in a single point), where $\varphi_T(p)$ is the point in $\gamma$ for which we have $W^s(y_l)\cap W^u(\varphi_T(p))$. Therefore, since $y_l$ is in a Pesin Block, we can define the weak-stable holonomy in a positive $m_{y_l}^u-$measure set which intersects $A_{y_l}\cap W_{loc}^u(y_l)$ is a positive $m_{y_l}^u-$measure set. Since $f^+$ is constant along weak-stable manifolds and by absolute continuity, then the image of this intersection by the weak-stable holonomy is a set of $m_{\varphi_T(p)}-$positive measure with the same Birkhoff average as $y_l$. Call this set $B_{\varphi_T(p)}$ and notice that by $s-$saturation and invariance, we get $B_{\varphi_T(p)}\subset M_0$.

\noindent Proceeding analogously with $x_k$ we produce a set $B_{\varphi_{T'}(p)}\subset M_0\cap W^u(\varphi_{T'}(p))$ with $m_{\varphi_{T'}(p)}^u-$positive measure and same Birkhoff average as $x_k$. 

\noindent Now, since $\gamma$ is periodic we take $\hat{T}>0$ such that $\varphi_{T+\hat{T}}(p)=\varphi_{T'}(p)$, then $\varphi_{\hat{T}}(B_{\varphi_T(p)})$ is a set with same Birkhoff average as $y_l$ of $m_{\varphi_{T'}}^u-$positive measure. Finally, $B_{\varphi_{T'}(p)}\subset M_0$ and then there exist points in $\varphi_{\hat{T}}(B_{\varphi_T(p)})$ with the same Birkhoff average as $x_k$. We conclude that $f^+(x)=f^+(x_{k})=f^+(y_{l})=f^+(y)$. Thus the restriction of $f^+$ to $\Delta$ is constant.  
\begin{figure}[h!]
\begin{overpic}[width=0.8\textwidth,tics=10]{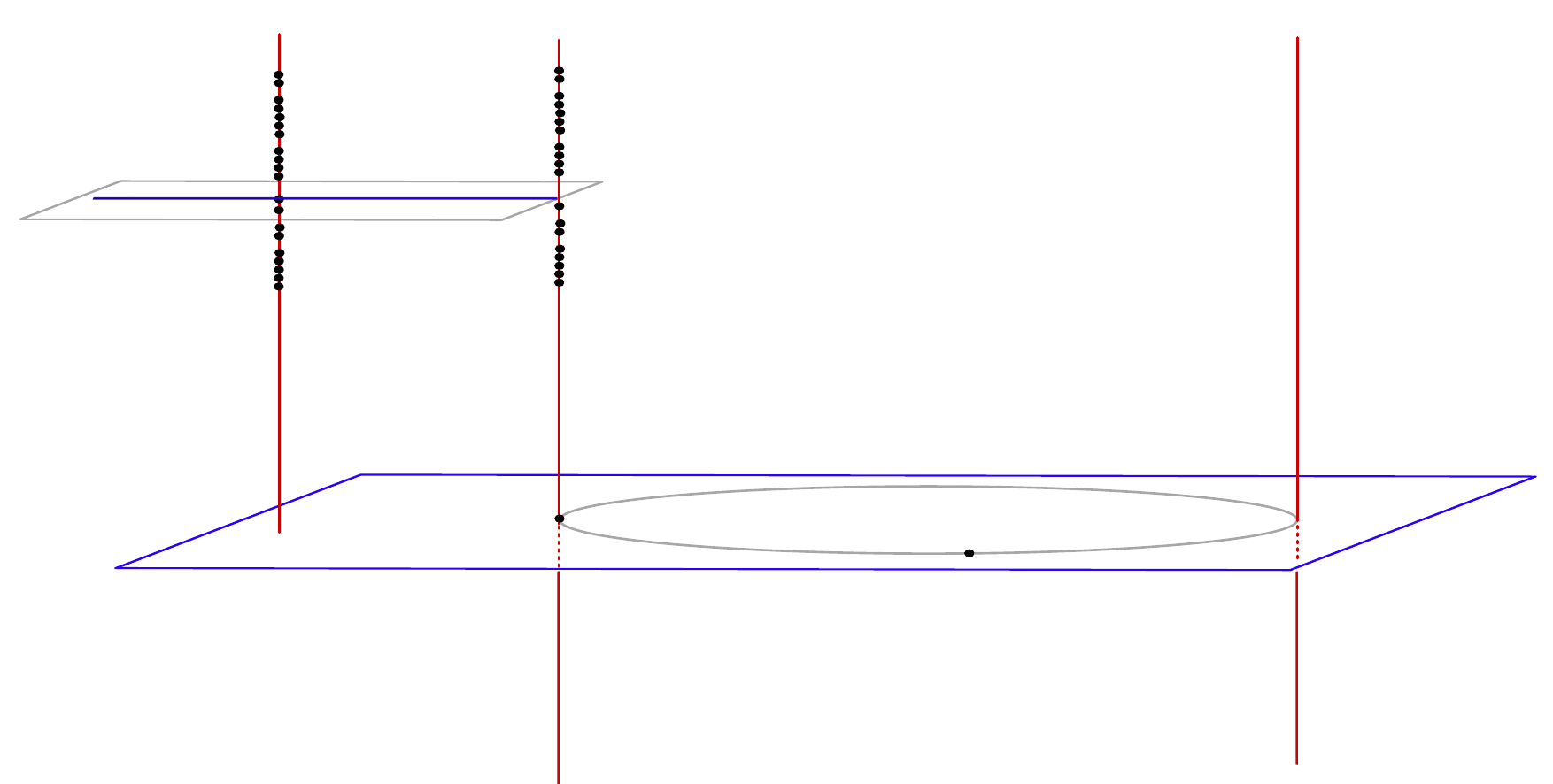}
\put(18,39){ $y_l$} 
\put(63,16){$p$}
\put(85,18){\textcolor{gray}{$\gamma$}}
\put(85,45){\textcolor{red}{$W^u(\gamma)$}}
\put(95,15){\textcolor{blue}{$W^s(\gamma)$}}
\put(9,25){\textcolor{red}{$W^u(y_l)$}}
\put(0,45){$A_{y_l}\cap W_{loc}^u(y_l)$}
\put(38,45){$B_{\varphi_T(p)}$}
\put(28,16){$\varphi_T(p)$}
\put(39,39){\textcolor{gray}{$W^{ws}(y_l)$}}
\end{overpic}
\centering
\caption{Construction of the set $B_{\varphi_T(p)}$}

\end{figure}

\subsection{Proof of the first Statement}

\noindent Consider the integrable function $f=1_{\Lambda^s(\gamma)}$ and the full-measure set of typical points $\mathcal{T}_f$ given by Lemma \ref{Typical points for integrable functions}. Define the $\varphi_t-$invariant set $\Delta^u = \Lambda^u(\gamma)\cap \mathcal{T}_f$. We will see that $\Delta^u \subset \Lambda^s(\gamma).$ 

\noindent Let $x \in \Delta^u$, then $W^u(x) \pitchfork W^s(\gamma) \neq \emptyset$. The strategy is to use an auxiliary point to construct a subset of $W^{wu}(x)$ intersecting $\Lambda^s(\gamma)$ with $m^{wu}_x-$positive measure. 

\noindent Let $\mathcal{R}_{\ve,l}^{i,j}$ be a Pesin Block for which the set $\Delta^s:=\Lambda^s(\gamma)\cap \mathcal{T}_f\cap R_{\varepsilon, l}^{i,j}$ has positive measure and consider an auxiliar point $y \in \Delta^s$ such that there is a sequence of times $(s_l)_l$ converging to infinity satisfying $y_l:=\varphi_{s_l}(y)\in \Delta^s$ and all $y_l$ are density points. Choose a point $z$ the point in $W^s(y)\pitchfork W^u(\gamma)$ and consider $z_l=\varphi_{s_l}(z)\in W^s(y)\pitchfork W^u(\gamma)$. Inside $R_{\varepsilon, l}^{i,j}$ the Pesin's manifolds have uniform lower bound for the diameter, say $\mathrm{diam}(W_{loc}^s(q))\geq \delta>0$, for all $q \in R_{\varepsilon, l}^{i,j}$. By definition we have that $d(z_l,y_l)\rightarrow 0$, as $l\rightarrow \infty$ so we can choose $z_0=\varphi_{s_{l_0}}(z)$ such that $d(z_0,y_0)<\frac{\delta}{2}$, where $y_0=\varphi_{s_{l_0}}(y)$. 

\noindent Since $y_0$ is a Lesbegue density point of $\Delta^s$, we can find a small ball $B$ centered in $y_0$ satisfying the condition $m(\Delta^s \cap B)>0$ and we can assume its diameter is equal to $\delta$. Let $\mathcal{F}$ be a smooth foliation of $B$ with dimension $n-\mathrm{dim} (W^s(y_0))$ such that each leaf $\mathcal{F}(\xi)$ is transverse to $W_{loc}^s(y_0)$. By the Fubini's property
\[
0<m(\Delta^s\cap B)=\int_{W_{loc}^s(y_0)}m_\xi^{\mathcal{F}}(\mathcal{F}(\xi) \cap \Delta^s)dm_{y_0}^s(\xi),
\]
we have that $m_\xi^{\mathcal{F}}(\mathcal{F}(\xi) \cap \Delta^s)>0$ on a subset of $W_{loc}^s(y_0)$ of $m_y^s-$positive measure. Fix $\xi_0$ such that $m_{\xi_0}^{\mathcal{F}}(\mathcal{F}(\xi_0)\cap \Delta^s)>0$.

\noindent We now prove that for some iterate of $x$ we get a transversal intersection between $W^{wu}(\varphi_T(x))$ and $W_{loc}^s(y_0)$. To do that, we want to see that we can approximate $W^u(\gamma)$ by iterating $W^u(x)$. In fact, since $W^u(x)\pitchfork W^s(\gamma)$, we divide into two cases. First, if $\mathrm{dim}W^u(x)+\mathrm{dim}W^s(\gamma)=n$, we can call $\{w\}=W^u(x)\cap W^s(\gamma)$ and consider a small saturation in the flow direction
\[
D:=\bigcup_{t\in(-\ve,\ve)}W^u(\varphi_t(x))
\]
Now, $D$ is a small disk transverse to $W^s(q)$, for some $q\in \gamma$ and we can apply the $\lambda-$lemma for the diffeomorphism $\varphi_T$, where $T$ is the period of $\gamma$. This implies that $\varphi_{nT}(D)$ converges in the $C^1-$topology to
\[
D'=\bigcup_{t\in(-\ve,\ve)}W^u(q).
\]
By taking big enough iterates we can make $\varphi_{nT}(D)$ close enough to $D'$ such that after flowing by $\varphi_t$ for some $0\leq t\leq T$ it "enters" $B$ and since $W_{loc}^s(y_0)\pitchfork W^u(\gamma)$ it must intersect $W_{loc}^s(y_0)$ transversally. The case where $\mathrm{dim}W^u(x)+\mathrm{dim}W^s(\gamma)>n$ is easier to deal with, because we already have that $W^u(x)\pitchfork W^s(q)$ for some $q\in \gamma$, so we just  apply the $\lambda-$lemma for $\varphi_{nT}$ and $W^u(x)$ directly to get that $W_{loc}^s(y_0)\pitchfork W^u(\varphi_t(x))$. So, let us fix that $x_t:=\varphi_t(x)$ satisfies $W_{loc}^s(y_0)\pitchfork W^{wu}(x_t)$ or $W_{loc}^s(y_0)\pitchfork W^{u}(x_t)$

\noindent To proceed we need to split each case into two more cases but the analysis is very similar. In the first case we deal with $\mathrm{dim}(W^s(y_0))+\mathrm{dim}(W^{wu}(x_t))=n$ (see Figure \ref{case with good dimensions}) and we consider the holonomy map $h$ between a subset of positive measure in $\mathcal{F}(\xi_0)$ and $W^{wu}(x_t)$, then $h$ sends the positive measure set $\mathcal{F}(\xi_0)\cap \Delta^s$ into a set of $m^{wu}_{x_t}-$positive measure. Since $\Lambda^s(\gamma)$ is $s-$saturated, the last set lies inside $\Lambda^s(\gamma)$. The second case we need to consider occurs when $\mathrm{dim}(W^s(y_0))+\mathrm{dim}(W^{wu}(x_t))>n$. In this case, we can assume with no lost of generality that we had choose the foliation $\mathcal{F}$ to be such that each leaf containing a point of $W^{wu}(x_t)$ is contained in $W^{wu}(x_t)$. Consider $S$ to be an open submanifold inside $W^s(y_0)\cap W^{wu}(x_t)$. Integrating over $S$ and using the holonomy maps from $\mathcal{F}(\xi_0)$ to $\mathcal{F}(q)$, for each $q \in S$, we construct a set $A$ intersecting $\Lambda^s(\gamma)$ with $m^{wu}_{x_t}-$positive measure. In any of those cases, using that $\Lambda^s(\gamma)$ is $\varphi_t-$invariant, we found a set of $m^{u}_{x_t}-$positive measure inside $\Lambda^s(\gamma)$. It means that $W^{u}(x_t)$ contains a set of positive measure with Birkhoff average $f^+ \equiv 1$, but $x_t$ is a typical point, so $f^+(x)=f^+(x_0)=1$. We conclude that $x \in \Lambda^s(\gamma)$. The analysis for the case $W_{loc}^s(y_0)\pitchfork W^{u}(x_t)$ is exactly the same.

\noindent The proof of the reverse inclusion is analogous, so we conclude the proof of the theorem.

\begin{figure}[h!]
\begin{overpic}[width=0.8\textwidth,tics=10]{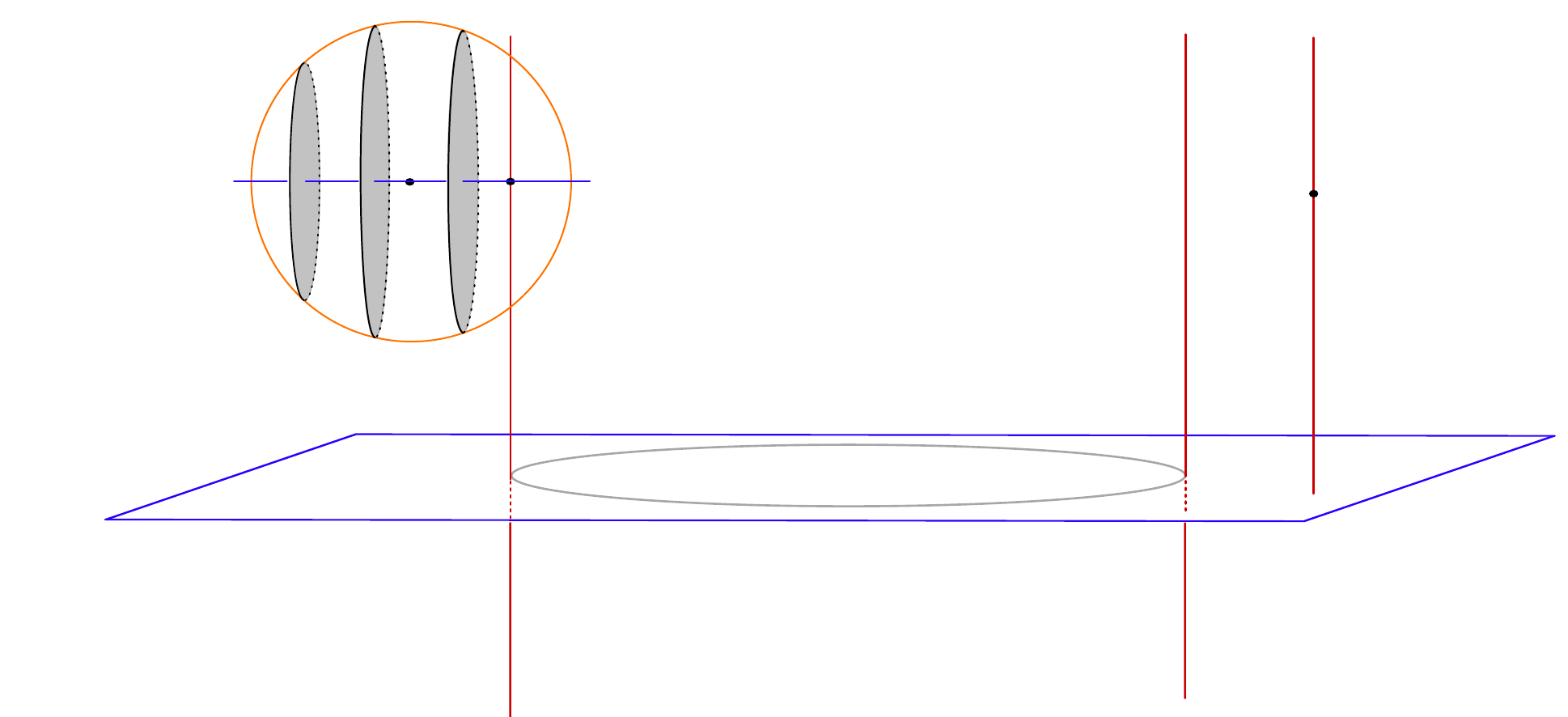}
\put(25,36){ $y_0$} 
\put(65,15){\textcolor{gray}{$\gamma$}}
\put(77,5){\textcolor{red}{$W^u(\gamma)$}}
\put(95,14){\textcolor{blue}{$W^s(\gamma)$}}
\put(4,33){\textcolor{blue}{$W_{loc}^s(y_0)$}}
\put(85,33){$x$}
\put(14,41){\textcolor{orange}{$B$}}
\put(21,21){\textcolor{darkgray}{$\mathcal{F}(\xi_0)$}}
\put(85,42){\textcolor{red}{$W^u(x)$}}
\end{overpic}
\centering
\caption{Case with $\mathrm{dim}W^s(y_0)+\mathrm{dim}W^u(\gamma)=n$ and $\mathrm{dim}W^u(x)+\mathrm{dim}W^s(\gamma)=n$.}
\label{case with good dimensions}
\end{figure}

\section{Proof of Theorem \ref{homoclinic class for SRB}}
\label{proof ergodicity SRB}
\noindent The proof will be analogous to the proof of the first statement of Theorem \ref{criteria for ergodicity}. Consider $f=1_{\Lambda^s(\gamma)}$ and let $\mathcal{T}_f$ be as in Lemma \ref{Typical points for L1 functions and SRB measures} and define $\Delta^u:=\Lambda^u(\gamma)\cap\mathcal{T}_f$. We are going to show that $\Delta^u\subset \Lambda^s(\gamma)$. To do that, let us prove that for every $x\in \Delta^u$ we get $f^+(x)=1$. For now on fix $x\in \Delta^u$. The strategy is the same as before, we will consider an auxiliary point $y$ to construct some local foliation of a small ball such that we find a leaf with intersection of positive measure with $\Lambda^s(\gamma)$, then we transfer this information to the unstable manifold of $x$ also via holonomy. The only difference here is the choice of the point $y$. Insted of a Lesbegue density point we choose $y$ to be a point in the support of the measure. The rest of the argument follows identically. 

\noindent Let $R_{\varepsilon,l}^{ij}$ be a Pesin block such that $\Delta^s:=R_{\varepsilon,l}^{ij}\cap \Lambda^s(\gamma)$ has positive measure and  $y\in \Delta^s$ be such that there exists a sequence $(t_k)_k$ converging to infinity such that $y_k=\varphi_{t_k}(y)\in \Delta^s$ and $y_k$ belongs to the support of $\mu$ restricted to $\Lambda^s(\gamma)\cap R_{\varepsilon,l}^{ij}$. Again, for points in $R_{\varepsilon,l}^{ij}$ their Pesin manifolds have a diameter bigger than a uniform constant $\delta>0$. As before, let $y_0=\varphi_{t_{k_0}}(y)$ be such that $d(y_0,W^u(\gamma))<\frac{\delta}{2}$.
Since $y_0$ is the support of $\mu$, we can find a small ball $B$ centered in $y_0$ satisfying the condition $\mu(\Delta^s \cap B)>0$. Let $\mathcal{F}$ be a smooth foliation of $B$ with dimension $n-\mathrm{dim} (W^s(y_0))$ such that each leaf $\mathcal{F}(\xi)$ is transverse to $W_{loc}^s(y_0)$. By the Fubini's property
\[
0<\mu(\Delta^s\cap B)=\int_{W_{loc}^s(y_0)}\mu_\xi^{\mathcal{F}}(\mathcal{F}(\xi) \cap \Delta^s)d\mu_{y_0}^s(\xi),
\]
we have that $\mu_\xi^{\mathcal{F}}(\mathcal{F}(\xi) \cap \Delta^s)>0$ on a subset of $W_{loc}^s(y_0)$ of $\mu_y^s-$positive measure. Fix $\xi_0$ such that $\mu_{\xi_0}^{\mathcal{F}}(\mathcal{F}(\xi_0)\cap \Delta^s)>0$. Since $\mu$ is a hyperbolic measure and has absolutely continuous conditional measures, the proof follows from the same arguments. 
We conclude that $\mu$ is hyperbolic and ergodic SRB measure, then the  physical property follows from \cite{PUGH&SHUB89ergodic}.

\section{Proof of Theorem \ref{SRB ergodic components has full measure homoclinic class}}
\label{proof SRB ergodic componets has full measure}
\noindent The proof will be a consequence of Theorem \ref{Katok existence of periodic orbit} and the next Lemma:
\begin{lemma}
    If $\mu$ is a regular, hyperbolic and SRB measure, then almost every one of its ergodic components is regular, hyperbolic and SRB.
\end{lemma}
\begin{proof}
    By the Ergodic Decomposition Theorem, there exist a partition $\mathcal{P}$ of $M$, a system of ergodic probability measures $\{\mu_P: P\in \mathcal{P}\}$, which form a decomposition for $\mu$, and a measure $\hat{\mu}$ over $\mathcal{P}$. Since $\mu$ is regular, then
    \[
    0=\mu(Sing(X))=\int\mu_P(Sing(X))d\hat{\mu}(P).
    \]
    It follows that $\hat{\mu}-$almost every $\mu_P$ is regular. Now, we just need to check that $\hat{\mu}-$almost every $\mu_P$ satisfies Pesin's formula. Indeed, by the Margulis-Ruelle inequality and ergodicity we have
    \[
    h_{\mu_P}\leq \int \sum_{\lambda(x)>0}\lambda(x)d\mu_P(x)=\sum_{\lambda(\mu_P)>0}\lambda(\mu_P).
    \]
    On the other side, by Jacobs Theorem (cf. \cite{KerleyViana} Theorem 9.6.2) we get
    \[
    h_{\mu}=\int h_{\mu_P}d\hat{\mu}(P),
    \]
    and since $\mu$ is an SRB measure,
    \[
    h_{\mu}= \int\sum_{\lambda(x)>0}\lambda(x)d\mu=\int\left(\int \sum_{\lambda(x)>0}\lambda(x)d\mu_P\right)d\hat{\mu}(P)=\int\sum_{\lambda(\mu_P)>0}\lambda(\mu_P)d\hat{\mu}(P).
    \]
    It follows that 
    \[
    \int \underbrace{\left(h_{\mu_P}-\sum_{\lambda(\mu_P)>0}\lambda(\mu_P)\right)}_{\leq 0}d\hat{\mu}(P)=0,
    \]
    hence $\hat{\mu}-$almost every $\mu_P$ satisfies the Pesin's formula. 
\end{proof}
\noindent To conclude the proof of Theorem \ref{SRB ergodic components has full measure homoclinic class}, notice that by the previous Lemma we can assume that $\mu$ is an ergodic measure. Now Theorem \ref{Katok existence of periodic orbit} implies that there exists a periodic hyperbolic orbit  $\gamma$ such that $supp(\mu)\subset \overline{\Lambda(\gamma)}$ and, in particular, for $\mu-$almost  every $x$ satisfies $\varphi_t(W^u(x))$ accumulates on $W^u(\gamma)$ when $t$ goes to infinity and $\varphi_t(W^s(x))$ accumulates on $W^s(\gamma)$ as $t$ goes to minus infinity. The invariance of $\Lambda(\gamma)$ implies that $\mu(\Lambda(\gamma))=1$.

\section{Proof of Theorem \ref{SRB measures with full measure homoclinic class are equal}}
\label{proof SRB measures with full measure homoclinic class are equal}
\noindent The strategy here is to prove that in this situation, we must have intersection of the Basins of $\mu$ and $\nu$, thus ergodicity will imply that they must coincide. 

Let $B(\mu)$ and $B(\nu)$ be the basins of $\mu$ and $\nu$, respectively. Since both measures are ergodic, it implies that $\mu(B(\mu))=1$ and $\nu(B(\nu))=1$. By Birkohoff Ergodic Theorem we can define the following $\mu-$full measure and $\nu-$full measure sets, respectively, 
    \[
    B_\mu = \left\{x\in M: \lim_{T\rightarrow \pm\infty }\frac{1}{T}\int_{0}^T f(\varphi_{t}(x))dt=\int fd\mu, \forall f\in C^0(M) \right\}
    \]
    \[
    B_\nu = \left\{x\in M: \lim_{T\rightarrow \pm\infty }\frac{1}{T}\int_{0}^T f(\varphi_{ t}(x))dt=\int fd\nu, \forall f\in C^0(M) \right\}.
    \]
    By hypothesis we get $\mu(B_\mu\cap\Lambda(\gamma))=1$ and $\nu(B_\nu\cap\Lambda(\gamma))=1$. Apply Lemma \ref{Typical points for L1 functions and SRB measures} for $\mu$ and the function $f=1_{B_\mu\cap\Lambda(\gamma)}$ to find $x\in \mathcal{T}_f\cap B_\mu\cap\Lambda(\gamma)$ such that $\mu_x^u(B_\mu\cap\Lambda(\gamma))=1$. Since $\mu_x^u$ is an SRB measure, its absolute continuity property implies that $m_x^u(B_\mu\cap\Lambda(\gamma))=1$. Analogously, we find $y\in\mathcal{T}_{f'}\cap B_\nu\cap\Lambda(\gamma)$ such that $m_y^u(B_\nu\cap\Lambda(\gamma))=1$, where $f'=1_{B_\nu\cap \Lambda(\gamma)}$.
    \noindent The strategy now is to transfer these sets to the unstable manifold of the orbit $\gamma$. To be precise, define the following sets
    \[
    D_x=B_\mu\cap \Lambda(\gamma)\cap W^u(x),
    \]
    \[
    D_y=B_\nu\cap \Lambda(\gamma)\cap W^u(y).
    \]
    Fix the points $p,q\in\gamma$ such that $p$ and $q$ are the points satisfying $W^s(x)\cap W^u(p)\neq \emptyset$ and $W^s(x)\cap W^u(p)\neq \emptyset$. By the definition of $\Lambda(\gamma)$ and using the absolute continuity of weak-stable holonomies, for every point $z \in D_x$ we find a unique point $w\in W^u(p)$ such that $\{w\}=W^{ws}(z)\pitchfork W^u(p)$. Since $B_\mu\cap\Lambda(\gamma)$ is an invariant and $s-$saturated set, we get that $m_p^u(B_\mu\cap \Lambda(\gamma))=1$. Now, take $T\in\mathbb{R}$ such that $\varphi_T(q)=p$ and use the same argument to $D_{\varphi_T(y)}$ to get  $m_p^u(B_\nu\cap\Lambda(\gamma))=1$. This implies that $B(\mu)\cap B(\nu)\neq \emptyset$, hence $\mu=\nu$.

\textbf{Acknowledgments:} The authors thank Prof. Ali Tahzibi for several fruitful discussions and for suggesting the study of SRB measures in this context. The first author also thanks the hospitality of the mathematics department from Penn State University during his visit when part of this work was written, in particular he would like to thank Prof. Federico Rodriguez-Hertz for his valuable guidance, advices and several interesting discussions. 

\bibliographystyle{plain}
\bibliography{Referencias2}

\Addresses

\end{document}